\documentclass[a4paper,11pt]{article}
\usepackage[utf8]{inputenc}

\usepackage{ifthen}
\newboolean{anon}
\setboolean{anon}{false}
\title{Local obstructions in sequences revisited}
\author{
  \ifthenelse{\boolean{anon}}{
    Authors
  }{
    Matthieu Rosenfeld, Alexander Shen\\ LIRMM, Univ Montpellier, CNRS, Montpellier, France
  }}
\date{}

\usepackage[dvipsnames]{xcolor}
\usepackage[color=Apricot]{todonotes}
\usepackage{amssymb,amsmath,amsthm}
\newtheorem{theorem}{Theorem}

\newtheorem{lemma}{Lemma}

\theoremstyle{remark}
\newtheorem{remark}{Remark}

\newcommand{\A}{\mathcal{A}}
\newcommand{\F}{\mathcal{F}}
\let\ge=\geqslant
\let\le=\leqslant
\begin{document}

\maketitle
\begin{abstract}
In this article we consider some simple combinatorial game and a winning strategy in this game. This game is then used to prove several known results about non-repetitive sequences and approximations with denominators from a lacunary sequence. In this way we simplify the proofs, improve the bounds and get for free the computable versions that required a separate treatment.
\end{abstract}

\section{Introduction}

Fifty years ago Erdős~\cite[p.~96]{erdosDiophantine} asked the following question. Let $n_k$ be a lacunary sequence of positive integers, i.e., $n_{k+1}>(1+\varepsilon)n_k$ for some $\varepsilon>0$. Can we find a real $\theta$ such that all distances from $n_k\theta$ to the nearest integer are separated from $0$ (exceed some positive $\delta$ for all $k$)? Several authors then showed that this is indeed the case, and provide some lower bounds for $\delta$ as a function of $\epsilon$. The bounds improved over time, and the best known asymptotic bound $\delta=\Omega\left(\varepsilon/|\log\varepsilon|\right)$ was obtained by Peres and Schlag~\cite{Peres2010Apr}; the computable version of this result was recently proved by Mourad~\cite{Mourad2024Jun}. Peres ans Schlag proofs relies on (some variant of) the Lovász Local Lemma, while Mourad's proof relies on his computable version of the Lefthanded Local Lemma.

We provide a simple proof of these results that uses some combinatorial game and a winning strategy in this game (based on Miller's argument in~\cite{Miller2012}). It gives slightly better constants in the bounds and does not require separate argument for the computable version.

The same game can be applied to sequences without repetitions. A classical (1906) result of Thue~\cite{thue06} says that there exists an infinite sequence in a three-letter alphabet that is square-free (does not contain substrings $uu$ for non-empty strings $u$). The same question may be asked in a more general situation where the alphabet is large (even infinite) but for each position only $3$ letters are allowed (the list of the allowed letters may depend on the position). The existence of such a sequence remains a conjecture~\cite[p.~215, Conjecture 1]{Grytczuk2013Mar}, but for lists of size $4$ the corresponding result is proven~\cite{Grytczuk2011Jan,Grytczuk2013Mar}, and the computable version of this result is proven~\cite{Mourad2024Jun} for lists of size~$6$. We note that all these results are easy corollaries of the game analysis, and the computable version does not require special treatment and is valid already for lists of size $4$.

The paper is organized as follows. In Section~\ref{sec:thegame} we explain the rules of the game and a sufficient condition for the existence of a winning strategy (Theorem~\ref{thm:thegame}). We note that this condition can be easily adapted to prove the existence of a \emph{computable} winning strategy (Theorem~\ref{thm:computablegame}). Then in Section~\ref{sec:miller} we apply this game to the question about forbidden substrings considered by Miller in~\cite{Miller2012} (our strategy is a straightforward generalisation of Miller's argument used there). In Section~\ref{sec:nonrepetitive} we observe that the game can be used to prove the existence of square-free sequences for $4$-list assignments. Then in Section~\ref{sec:denominators} we show how the game can be applied to Erdős's question. Finally, in Section~\ref{sec:history} we provide some historical context and discuss the game approach.

\section{The game}\label{sec:thegame}

We consider a combinatorial game that appears to be the core of several existence proofs. 

Alice and Bob play an infinite game on the $k$-ary rooted tree. Initially Bob is in the root of this tree. His goal is to escape (to infinity) along some branch of the tree, one edge at a time. Alice tries to block Bob by declaring some tree vertices as \emph{forbidden} for him. 

Alice and Bob alternate; Alice starts the game.  At each step Alice adds some vertices to the set $F$ of forbidden vertices (initially empty); then Bob moves along some tree edge (moving from a vertex to its child). Bob loses if he gets into a forbidden vertex. Alice loses if Bob never gets into a forbidden vertex. Of course, without additional restrictions Alice wins the game: she may forbid all vertices, or all children of current Bob's vertex, or all vertices of some level ahead of Bob. To make the game non-trivial, we will require that Alice does not forbid too many vertices too close to Bob.

Before formulating the restrictions, let us note that the game can be reformulated as follows. First, only the subtree rooted at Bob's current position matters (other vertices are inaccessible to Bob anyway). So we may think that at his turn Bob chooses an \emph{available} ($=$ not forbidden) child of the current position and the game continues on the corresponding subtree. At her turn Alice adds some vertices of this subtree to $F$ (making them \emph{unavailable} to Bob); the vertices forbidden earlier remain forbidden.  Alice wins if at some point Bob has no available child to choose from, and Bob wins otherwise. (We also assume that Alice cannot forbid Bob's current position, i.e., the root of the current subtree.)

To formulate the restrictions for Alice, let us fix some $\beta$ in the range $(1,k)$. Each vertex of the $l$-th level of the $k$-ary rooted tree has weight $\beta^{-l}$ (so the root has weight $1$, each its child has weight $1/\beta$, etc., cf. the definition of Hausdorff's outer measure). The weight of a set of vertices is the sum of their weights. The restriction for Alice is as follows: \emph{for every move the weight of the newly forbidden vertices, measured in the current subtree, does not exceed some constant $\omega$}. In this way, for every choice of parameters $\beta$ and $\omega$, we get a game on a $k$-ary tree called $(\beta,\omega,k)$-game in the sequel.

\begin{theorem}\label{thm:thegame}
Assume that for some integer $k>1$, some $\beta$ in $(1,k)$ and some $\omega>0$ the inequality
\begin{equation}\label{maincond}
\beta (1+\omega) \leqslant k
\end{equation}
holds. Then Bob has a winning strategy in the $(\beta,\omega,k)$-game.
\end{theorem}

\begin{proof}
The winning strategy for Bob is to maintain the following invariant: after his move (and before the next move of Alice) \emph{the total weight of the forbidden vertices in the current subtree} (rooted at Bob's position) \emph{is less than $1$}. (This implies that the current position is not forbidden.)

Initially (before the first move of Alice) this is true, since the weight is $0$. Let us check that Bob can maintain this invariant. After Alice adds new forbidden vertices (with total weight at most $\omega$) the total weight of the forbidden vertices in the current subtree is less than $1+\omega$, and the root of the subtree is not forbidden. So all the forbidden vertices are split among $k$ children of the root, and there is a child whose share is less than $(1+\omega)/k$. Bob moves to this child; this increases all the weights in its subtree by factor $\beta$ since the obstacles become one step closer,  and we note that
\(
\beta(1+\omega)/k \le 1
\)
by assumption.
\end{proof}

We can use a computable version of this result to construct a computable path in the tree by playing against the computable strategy for Alice (for computable $\beta$ and $\omega$). More precisely, the next move of Alice should consist of programs that compute
\begin{itemize}
\item a set $F$ of newly forbidden vertices;
\item a modulus of convergence for their total weight.
\end{itemize}
The first program says for each vertex whether it becomes forbidden, and the second program gives for every rational $\varepsilon>0$ some integer $N$ such that the total weight of all $F$-vertices of height greater than $N$ is less than $\varepsilon$.

A special case: Alice explicitly gives a finite list of newly forbidden vertices.

\begin{theorem}\label{thm:computablegame}
Assume that $\beta$ and $\omega$ are computable real numbers satisfying the same condition as in Theorem~\ref{thm:thegame}, and Alice uses a computable strategy. Then Bob can win against this strategy using a computable path in the $k$-ary tree.
\end{theorem}

\begin{proof}
Bob can maintain the invariant as before, keeping the total weight of the forbidden vertices in the subtree as a computable number. Since Alice behaves computably (in the sense explained above), Bob can update this data after her move. Then he computes the weights for the children until he finds a good child (note than the inequality is strict, so such a child can be found effectively if it exists).
\end{proof}

Note that, since the only thing needed to compute Bob strategy is to be able to compute the weight of the set of forbidden vertices, if we can do so quickly then we can quickly compute Bob's move.

\section{Miller's argument}\label{sec:miller}
The idea behind this technique of $\beta$-weights was first used by J.~Miller~\cite{Miller2012} to prove the existence of an infinite sequence avoiding some set of forbidden factors (contiguous substrings). Let us present his proof as an application of Theorem~\ref{thm:thegame}.

\begin{theorem}\label{thm:miller}
  Let $\A$ be a finite alphabet and let $\F\subseteq\A^+$ be a set of $\A$-words. Suppose there exists $\beta$ in $(1,|\A|)$ such that 
  \[
     \beta\left(1 + \sum_{f\in\F}\beta^{-|f|}\right)\le |\A|\,.
  \]
  Then there exists an infinite sequence in $\A^\mathbb{N}$ that avoids $\F$.
\end{theorem}

\begin{proof}
Let $k$ be the size of $\A$. Consider the $k$-ary tree where children are indexed by letters. Bob constructs a path in this tree, and at every move Alice forbids all elements of $\F$ starting from Bob's position.
\end{proof}

The computable version can be obtained in the same way using Theorem~\ref{thm:computablegame}:

\begin{theorem}
  Let $\A$ be a finite alphabet and let $\F\subseteq\A^+$ be a computable set of $\A$-words. Suppose there exists a computable $\beta\in (1,k)$ such that 
  \[
    \beta\left(1 + \sum_{f\in\F}\beta^{-|f|}\right)\le |\A|\,
  \]
 and the series in the left hand side converges computably. Then there exists an infinite computable sequence in $\A^\mathbb{N}$ that avoids $\F$.
 \end{theorem}
 
 The construction given by Miller~\cite[Proposition 2.1]{Miller2012} is exactly the same (he uses letter $S$ instead of $\F$, as well as $n$ instead of $k$, and $c$ instead of $1/\beta$ in our notation). However, extracting a combinatorial game from this argument allows us to apply the same technique to several other questions considered in the following sections.
  
\section{Nonrepetitive words}\label{sec:nonrepetitive}
A \emph{square} is a word of the form $uu$ with $u$ a non-empty word, and a square-free word is a word that contains no squares as factors. For instance, \texttt{hotshots} is a square, \texttt{infinite} is square-free and \texttt{repetition} contains the square \texttt{titi}. In 1906, Thue proved that there exists an infinite square-free word over a three-letter alphabet ~\cite{thue06}.
This result is often called the first result in combinatorics on words and many its variants have been studied. Using Theorem \ref{thm:thegame}, we can construct square-free sequences for all~$k\ge4$. This is still weaker than Thue's result, but our argument can in fact be applied to the more general case where the alphabet is large (even infinite), but only four letters are allowed for each position.

More formally, consider an infinite alphabet $S$ and some \emph{$k$-list assignment}, i.e., a function $\alpha: \mathbb{N}\to \mathcal{P}(S)$ that for every $n\in \mathbb{N}$ specifies which $k$ letters from $S$ are allowed in $n$-th position (so $|\alpha(n)|= k$). We say that a sequence $w_0w_1\ldots$ \emph{respects} $\alpha$ if $w_i\in \alpha(i)$ for all $i$. In 2011, Grytczuk, Przyby{\l}o and Zhu~\cite{Grytczuk2011Jan} proved that for any $4$-list assignment $L$ there exists an infinite square-free sequence that respects $L$; several other proofs using different tools appeared later (see Section~\ref{sec:history}). Recently a computable version of this result was proven by Mourad~\cite{Mourad2024Jun} but only for $6$-list assignments. In this section we note that these results are easy corollaries of the game and get the computable version for $4$-lists as a byproduct (thus improving a bit the result from~\cite{Mourad2024Jun}). 

\begin{theorem}\label{thm:assignment}
For every $4$-list assignment $\alpha$ there exists an infinite square-free sequence that respects $\alpha$. If the assignment is computable, then there exists a computable square-free sequence that respects $\alpha$.
\end{theorem}

\begin{proof}
We consider the game over the $4$-ary tree that describes which letters can be added at each step according to $\alpha$. Every vertex at level $d$ has $4$ children labeled with letters in $\alpha(d)$. Vertices of this tree correspond to $S$-words that satisfy $\alpha$-restrictions, and infinite paths are sequences respecting~$\alpha$. 

Consider the following strategy for Alice. When Bob's position is some word $w$, she forbids all the vertices corresponding to words $wv$ for all $v$ that are non-empty suffixes of $w$ (thus preventing $vv$ from appearing in the sequence). In this way every possible non-empty square will be forbidden by Alice at the moment when the first half of this square appears.

To show that Bob has a winning strategy, let  $\beta=2$. The $\beta$-weight of Alice's move is at most 
\[
\frac{1}{2} + \frac{1}{4} + \ldots < 1
\]
since Alice forbids at most one string of each positive length. For $\omega=1$ the conditions of Theorems~\ref{thm:thegame} and~\ref{thm:computablegame} are fulfilled; indeed, $2(1+1)\le 4$.
\end{proof}

\begin{remark}
  In our setting for every position there are four letters allowed in this position. The same argument can be applied if the set of four allowed letters depends on the previous letters and not only of the height of the vertex (i.e., when we consider a $4$-ary tree whose edges are labeled by letters from~$S$).
\end{remark}  

Note that after $n$ steps there are at most $O(n^2)$ forbidden vertices in the current subtree, and they can easily be found in polynomial time in $n$ as long as the list of each position can be computed in polynomial time. Under this condition, the square-free sequence that respects the list assignment can be computed in polynomial (in~$n$) time.

\section{Separated occurrences}
The following result proven by Mourad~\cite{Mourad2024Jun} is a computable version of a result of Beck~\cite{Beck1984Jan} that was the historically first application of the Lovász Local Lemma to sequences. It can be seen as a generalization of Thue's result: now two occurrences of the same string not only cannot follow each other, but have to be far apart. 

\begin{theorem}[\cite{Mourad2024Jun}]\label{thm:computablebeck}
For every $c<2$ there exist a constant $N$ and a computable infinite binary sequence $\alpha$ such that for every string $x$ of length $n\ge N$ any two occurrences of $x$ in $\alpha$ are far away: the distance between them is at least~$c^n$.
\end{theorem} 

The distance between two occurrences starting at positions $i$ and $j$ is defined as $|i-j|$. Note that $c$ cannot be greater than $2$ since then we require more than $2^n$ substrings of length $n$ to be different.

\begin{proof}[Proof of Theorem \ref{thm:computablebeck}]

This result is also an easy application of the combinatorial game of Section~\ref{sec:thegame}. Bob constructs $\alpha$ bit by bit, and Alice does not allow $\alpha$ to violate the statement of the theorem. Namely, if $x$ is a substring already present in~$\alpha$ (not too long ago), Alice prevents $x$ from appearing again (starting from the current position). 
\begin{center}
\includegraphics{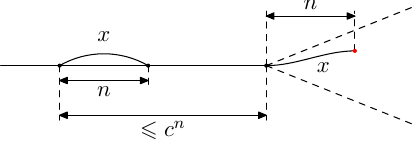}
\end{center}
In this way each move of Alice forbids at most $c^n$ vertices of level $n$ (for all $n\ge N$), and we know that Bob can win if
\[
\beta \left(1 + \sum_{n\ge N}\frac{c^n}{\beta^n}\right)\le 2
\]
for some $\beta\in (1,2)$. Note that this condition is satisfied if $\beta$ is chosen arbitrarily in $(c,2)$, and $N$ is large enough (so the tail of the geometric series is small).

This would finish the proof if not one minor problem that needs to be dealt with: the distance between two potential occurrences of $x$ may be small and Alice does not have at hand the string that she needs to prevent (the end of this string is not determined yet). 

\begin{center}
\includegraphics{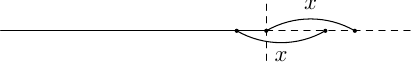}
\end{center}

Still she can reconstruct the string to be avoided: choosing some bit in the right occurrence of $x$, she already knows the corresponding bit in the left occurrence (either it was in the already constructed part of $\alpha$, or it was already chosen previously in the context of right occurrence).
\end{proof}

If we want to know how $N$ depends on $c$, we need to choose the optimal value of $\beta$ that makes the left hand side of the inequality minimal.
Note also that, as in the previous section, the result holds for the list setting and even if the lists are chosen by Alice dynamically. 

In this case, every forbidden vertex is uniquely determined by 3 positions from the current word (the start and end of the first occurrence of $x$, and the start of the second occurrence of $x$). So the number of forbidden vertices in the current subtree is bounded by a polynomial in the number of steps. We may assume without loss of generality that $c$ is a rational number (it can be increased if necessary), so the $n$-bit prefix of the sequence we construct can be computed polynomially (in $n$). 

\section{Making adjacent blocks very different}

Mourad also considered another computable generalization of Thue's result (non-computable version was first mentioned as an exercise in \cite{ProbabilisticMethod}).
\begin{theorem}[\cite{Mourad2024Jun}]\label{thm:computableExercise}
For any  $\varepsilon>0$, there is some $N$ and a computable infinite binary sequence $\alpha$ such that any two adjacent blocks in $\alpha$ of length $n\ge N$ differ in at least $(1/2-\varepsilon)n$ places. 
\end{theorem}

In the previous section we required two blocks to be different if they are not very far apart; now we consider two adjacent blocks but require them to be ``very different''. We provide a simple proof of this result based on our combinatorial game. 

\begin{proof}[Proof of Theorem \ref{thm:computableExercise}]
First we note that the fraction of $n$-bit strings where frequency of ones is less than $1/2-\varepsilon$ is exponentially small (Chernoff's bound).  So the number of $n$-bit strings that differ from a given string $x$ in less than $(1/2-\varepsilon)n$ places is bounded by $c^n$, where $c<2$ is some constant (depending on~$\varepsilon$).

Again Bob constructs $\alpha$ bit by bit and Alice forbids the undesirable continuations.  If Bob's position corresponds to a word $w$, then for all suffixes $v$ of $w$ of length $n\ge N$ Alice forbids not only the extension $v$ but all $v'$ of length $n$ that differ from $v$ in less than $(1/2-\varepsilon)n$ places (there are at most $c^n$ of them). Bob can win if
\[
 \beta \left(1+\sum_{n\ge N} \frac{c^n}{\beta^n}\right) \le 2
\]
and again for arbitrary $\beta\in(c,2)$ it happens for sufficiently large~$N$.
\end{proof}

This result is a corollary of Miller's theorem (Theorem~\ref{thm:miller}): we forbid all the factors $vv'$ where $v$ is long enough and $v'$ is close to $v$, and there are not many of them. But the game argument gives a bit more: it shows that the result holds also in the list setting, even if the lists are chosen at the last moment by Alice. It gives also a better bound on $N$ (as function of $\varepsilon$) since now we specify the forbidden extensions later (when the first half is already known).

Note that Alice adds exponentially many forbidden vertices at every step, so this algorithm does not imply immediately the existence of a polynomial time computable sequence with desired property. However, this can be easily fixed: the number of strings that differ from a given string at most by a given distance is a sum of (polynomially many) binomial coefficients and can be easily computed.  The distribution of these strings between children is also easy to compute (since the condition is again about distance). The only problem appears because the same string can be prohibited many times --- to avoid this problem, we may consider the \emph{multiset} of forbidden strings and compute the $\beta$-weights accordingly.

\smallskip

The game technique can be applied to other questions related to word combinatorics and repetitions;  it seems that many results about sequences obtained by the entropy compression technique (e.g. \cite{Grytczuk2013Mar, Ochem2014Apr,ChybowskaSokol2022Jul}), the power series method (e.g. \cite{Bell2007Sep,Rampersad2011Jun,Ochem2016Feb}) and the counting argument (e.g., \cite{Rosenfeld2021Oct}) can be proven using the game technique with the same bounds, but with the added benefit of computability (and simpler proofs). On the other hand, it is unclear whether the variant of the counting argument used in \cite{Rosenfeld2022Sep,Rosenfeld2024May} can be replaced by a game argument (it is also unclear whether the entropy-compression technique can be used in that case).

\section{Application to denominators}\label{sec:denominators}

In this section we apply the combinatorial game described in Section~\ref{sec:thegame} to a number-theoretic question and consider rational approximations with restricted denominators, giving a new proof of results from~\cite{Mourad2024Jun,Peres2010Apr} (we also slightly improve the bounds from both papers and provide a computable version of the result that appeared in~\cite{Mourad2024Jun}).

Let $\theta$ be some real number. To find its best rational approximation with a given (positive integer) denominator $t$, we multiply $\theta$ by $t$ and take the closest integer. The ``quality'' of approximation can be measured by $d(t\theta,\mathbb{Z})$, the distance between $t\theta$ and the nearest integer. This distance is a number between $0$ (ideal approximation) and $1/2$ (the worst case).  We will prove that \emph{for every sparse enough set of denominators there exists a real number that is hard to approximate using denominators from this set}, and that this number can be made computable if the set of denominators is computable. 

We will use the following terminology: a real $\theta$ is \emph{$\varepsilon$-regular for denominator $t$} if $d(t\theta,\mathbb{Z})> \varepsilon$. Here $\varepsilon>0$ is some threshold, and $t$ is a positive integer; this condition means that $\theta$ is ``not too easy to approximate with denominator $t$''.  Assume that a sparse set $T$ of positive integers (allowed denominators) is given. By sparsity we mean that that for every $j$ there are only $O(1)$ integers between $2^j$ and $2^{j+1}$ that belong to $T$. Then \emph{there exists some real $\theta$ and some positive $\varepsilon$  such that $\theta$ is $\varepsilon$-regular for all denominators $t\in T$}. This question (in a slightly different form, see Section~\ref{sec:history}) was first raised by Erdős~\cite{erdosDiophantine}.

Let us start with some simple remarks. Obviously for a given $t$ and for every $\varepsilon<1/2$ there are reals that are $\varepsilon$-regular for $t$. Moreover, for small~$\varepsilon$ most reals are $\varepsilon$-regular for $t$:  the reals that are \emph{not} $\varepsilon$-regular for $t$, form stripes around $k/t$ (for integer values of $k$) of width $2\varepsilon/t$ (going $\varepsilon/t$ left and $\varepsilon/t$ right). Those stripes cover $2\varepsilon$-fraction of all points. 
\begin{center}
\includegraphics{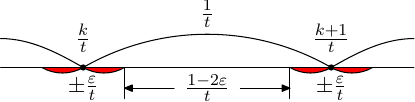}
\end{center}
The union bound then allows us to prove the existence of a real that is $\varepsilon$-regular with respect to several denominators $t_1,\ldots,t_k$, assuming that ${2k\varepsilon<1}$. 

This argument is not enough to construct a number that is $\varepsilon$-regular for infinitely many denominators (and some fixed $\varepsilon$).  And indeed it is not possible if $T$ contains all integers: the Dirichlet theorem says that for every $\varepsilon>0$ there is no real $\theta$ that is $\varepsilon$-regular for \emph{all} denominators. (Indeed, consider the sequence $n\theta \pmod 1$ and find two points $n_1\theta$ and $n_2\theta$ that are $\varepsilon$-close; then $t=n_2-n_1$ is a denominator that provides a good approximation, since $t\theta$ is $\varepsilon$-close to $0$.) This is why sparsity is needed. Informally speaking, it guarantees that the fragmentation happens on different levels: stripe patterns for small and large values of $t$ are (almost) independent, similar to the Lov\'asz local lemma (though it is not true literally:  e.g., stripes for $t$ are not independent with stripes for $2t$). We utilize this structure and construct a decreasing sequence of dyadic intervals whose common point is $\varepsilon$-regular for every $t\in T$. Here is the exact statement.

\begin{theorem}\label{thm:diophantine}
   Let $T$ be a set of positive real numbers such that 
    \[
      | T \cap (2^{j}, 2^{j+1}]| \le C
    \]
    for all integers $j$ and some constant~$C$. 
    Assume that for some $k\ge 3$ there exists $\beta\in (1,2)$ such that 
    \begin{equation*}
      \beta + \frac{3C}{\beta^{k-1}}\le2\,. \eqno(*)
    \end{equation*}
    Then there exists a real $\theta\in[0,1]$ that is $2^{-k}$-regular with respect to all $t\in T$.
    Moreover, if $T$ is computable then there exists a computable $\theta$ with this property.
\end{theorem}

Note that for large $k$ we can find $\beta$ that satisfies~$(*)$; moreover, for every $\beta\in (1,2)$ all sufficiently large $k$ satisfy this inequality. So indeed for every $T$ that contains at most $O(1)$ elements in every interval $(2^j,2^{j+1}]$ we can find some $\varepsilon>0$ and a real $\theta$ that is $\varepsilon$-regular with respect to all denominators $t\in T$. (To find how $\varepsilon$ depends on $C$, we should optimize the choice of $\beta$, see~Theorem~\ref{thm:bounds} below.)

\begin{proof}[Proof of Theorem~\ref{thm:diophantine}]
For any $\ell$, an \emph{$\ell$-dyadic interval} is an interval of the form $\left[\frac{u}{2^\ell},\frac{u+1}{2^\ell}\right]$ where $u$ is an integer. We consider a tree formed by dyadic intervals inside $[0,1]$. The root of this tree is the $0$-dyadic interval  $[0,1]$, and each vertex has two children (left and right half). An infinite branch in this tree is a sequence of nested dyadic intervals, and we will construct $\theta$ as the (unique) common point of such a sequence for a suitable chosen branch.

Let us split all the denominators we care about into groups:
\[
    P_j =\{t\in T \colon 2^{j} < t \le 2^{j+1}\}
\]
Here $j\ge -1$ is an integer (for $j=-1$ only $1$ can be in $P_j$). By assumption $|P_j|\le C$ for all $j$.  At every step of the construction we consider one group of denominators. Let us say that  an $\ell$-diadic interval $I$ is \emph{valid} if each of its point is $2^{-k}$-regular for all denominators $t\in P_{\ell-k}$. Note that $\ell - k$ is negative for small $\ell$ and the group $P_{\ell -k}$ may be undefined; in this case all $\ell$-diadic intervals are valid. (We will see later how this shift by $k$ helps.)

We consider the game between Alice and Bob. When Bob is at level $m$ of the tree, Alice forbids all the vertices of level $m+k-2$ that correspond to invalid dyadic intervals. (Again, we will see later why $k-2$ helps; note that according to our conventions we use denominators from  $P_{m-2}$ to define valid intervals of level $m+k-2$.)
 \begin{center}
 \includegraphics{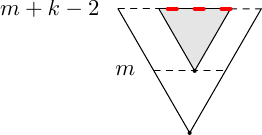}
\end{center}

\begin{lemma}\label{lem:bound}
At each step Alice forbids at most $3C$ vertices.
\end{lemma} 

\begin{proof}
Vertices are prohibited at step $m$ because the corresponding interval at level $m+k-2$ is not valid, i.e., contains some point that is close to a fraction with denominator in $P_{m-2}$. There are only $C$ denominators in $P_{m-2}$, so it is enough to show that each denominator causes at most $3$ forbidden intervals. We need to show that the ``bad stripes'' of width $2\varepsilon/t$ (where $\varepsilon=2^{-k}$) separated by gaps of length $(1-2\varepsilon)/t$ may intersect at most $3$ dyadic intervals of level $m+k-2$ that are parts of the dyadic interval at level $m$ (corresponding to the current position of Bob).

Consider some denominator $t$ in $P_{m-2}$, so $2^{m-2}<t\le 2^{m-1}$.  The distance between two bad stripes is $(1-2\varepsilon)/t$. Since $k\ge 3$, we have $\varepsilon=2^{-k}\le 1/8$; recalling that $t\le 2^{m-1}$ we see that the distance is bigger that $2^{-m}$ (the length of an $m$-dyadic interval):
\[
\frac{1-2\varepsilon}{t} \ge \frac{3/4}{2^{m-1}}>\frac{1}{2^m}.
\]
 Therefore only one stripe matters, and we need to check that it may intersect at most three $(m+k-2)$-dyadic intervals. For that it is enough to show that the width of the stripe is less than two times the length of the interval (if the stripe touches four intervals, it necessarily contains two),
 \begin{center}
 \includegraphics{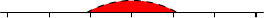}
\end{center}
and indeed
 \[
 \frac{2\varepsilon}{t}< \frac{2^{-k+1}}{2^{m-2}} = 2^{-k-m+3}=2\cdot 2^{-(m+k-2)}.
 \]
 Lemma~\ref{lem:bound} is proven.
\end{proof}

Now we combine the Lemma and the criterion from Theorem~\ref{thm:thegame} and show that Bob can win the game. Indeed, Alice forbids $3C$ vertices, each has weight $1/\beta^{k-2}$, so the inequality from Theorem~\ref{thm:thegame} requires that
\[
\beta\left(1+\frac{3C}{\beta^{k-2}} \right) \le 2,
\]
and that is what is guaranteed by the condition.
\end{proof}

Note that at step $n$ we have chosen $n$ bits of $\theta$ and consider denominators that are~$O(2^n)$. Since we add at most $3C$ forbidden vertices at every step, the set of forbidden vertices and its weight can be computed in time polynomial in $n$ (given the list of denominators as an oracle). So the $n$-th bit in the binary expansion of $\theta$ is polynomialy (in $n$) computable with the same oracle. 

\section{Optimizing the choice of parameters}

The parameter $\beta$ does not appear in the claim of Theorem \ref{thm:diophantine}, so for given $C$ and $k$ we may choose $\beta$ that makes the left hand side of the inequality (i.e., $\beta +3C/\beta^{k-1}$) minimal and verify that this minimal value does not exceed $2$. This minimal value is easy to find; for example, we may split $\beta$ into a sum of $k-1$ terms $\beta/(k-1)$ and note that we minimize the sum of $k$ quantities whose product is constant, so the minimal value is achieved when all the terms are equal:
\[
\frac{\beta}{k-1}=\frac{3C}{\beta^{k-1}}.
\]
This gives $3C(k-1)=\beta^k$, or $\beta = (3C(k-1))^{1/k}$, and the minimal value is 
\[
 \frac{k}{k-1}\beta=\frac{k}{k-1}(3C(k-1))^{1/k}.
 \]
 It does not exceed $2$ when
 \[
 \left(\frac{k}{k-1}\right)^k\cdot 3C(k-1) =  \left(\frac{k}{k-1}\right)^{k-1}\cdot 3Ck \le 2^k.
 \]
The first factor is an approximation (from below) to $e=2.71828\ldots$, so we get the following version of Theorem~\ref{thm:diophantine} where $\beta$ is eliminated:

\begin{theorem}\label{thm:diophantine-opt}
Let $T$ be a set of positive integers such that 
    \[
      | T \cap (2^{j}, 2^{j+1}]| \le C
    \]
 for all integers $j$ and some constant~$C$. Assume that for some $k\ge 3$ the inequality
    \[
     3eCk \le 2^k
    \]
holds. Then there exists a real $\theta\in[0,1]$ that is $2^{-k}$-regular with respect to all $t\in T$. Moreover, if $T$ is computable then there exists a computable $\theta$ with this property.
\end{theorem}

For example, for $C=1$ (minimal possible value) we may take $k=6$ and conclude that there exists a real that is $(1/64)$-regular with respect to all $t\in T$. For bigger values of $C$ we should take the minimal value of $k$ that fits the condition of Theorem~\ref{thm:diophantine-opt}; it guarantees the existence of a $\varepsilon(C)$-regular $\theta$ with respect to all denominators in $T$ for $\varepsilon(C)=2^{-k(C)}$.  Asymptotically (as $C\to\infty$) this value is 
\[
\varepsilon(C) = \frac{1}{(3e+o(1))C\log_2C}
\]
Indeed,
\[
k(C)\approx \log_2 (3eCk(C)) = \log_2(3eC)+\log_2k(C),
\]
so 
\[
k(C)=(1+o(1))\log_2(3eC)=(1+o(1))\log_2 C
\]
 and 
 \[
 \log_2k(C)=\log_2 (3eC) + \log_2((1+o(1))\log_2 C).
 \]
 If we want a bound that is valid for all $C\ge 2$, one can estimate the values of the parameters for small $C$ and see that 
 \[
   \theta(C)=\frac{1}{64C\log_2 C}
 \]
 is enough (here we use $64$ instead of the asymptotic value $3e=8.15\ldots$). This gives the following corollary:
 \begin{theorem}\label{thm:bounds}
 Let $C\ge 2$ be an integer and let $T$ be a \textup[computable\textup] set of positive integers that contains at most $C$ elements in every interval $(2^j,2^{j+1}]$. Then there exists a \textup[computable\textup] real $\theta$ that is $1/(64C\log_2C)$-regular for all $t\in T$.
 \end{theorem}
 
 This result slightly improves the bound from~\cite{Peres2010Apr} and its computable version in~\cite{Mourad2024Jun}. Moreover, if we can compute in time polynomial in $n$ the $n$ first bits of the $n$ first elements of $T$, then we can compute in time polynomial in $n$ the $n$ first bits of $\theta$.

\section{Historical remarks}\label{sec:history}

\subsection{Erdős' question}

The original formulation of Erdős' question in \cite[p.~96]{erdosDiophantine} was: ``Finally I state a few disconnected problems. Let $n_1 < n_2 < \ldots$ be an infinite sequence of integers satisfying $n_{k+1}/n_k> c > 1$. Is it true that there always is an irrational $\alpha$ for which the sequence $(n_k\alpha)$ is not everywhere dense?'' (The parentheses in $(n_k\alpha)$ mean that the sequence is considered modulo~$1$.) Note that Erdős' asks this sequence not to be dense, and we discuss a bit stronger requirement: the sequence is never $\varepsilon$-close to $0$. 

Another difference is that the condition for the denominators is stronger: Erdős requires that $n_{k+1}/n_k> c$ for some $c>1$ while we use only that each interval $(s,2s]$ contains $O(1)$ denominators.

There was a sequence of improvements related to Erdős' question that is summarized by Peres and Schlag~\cite{Peres2010Apr}: first bounds of Mathan (1980) and Pollington (1979) had $O(C^4\log C)$ in the denominator;  Katznelson (2001) improved the bound replacing the denominator by $O(C^2 \log C)$; then Akhunzhanov and Moshchevitin (2004) got $O(C^2)$ (no logarithmic factor),  and such a bound with better explicit constant was also obtained by Dubickas (2006). Peres and Schlag~\cite{Peres2010Apr} provide $O(C\log C)$ bound; 
Mourad used additional ideas to get the computable version (though with slightly worse constant $1/360$ instead of $1/240$ by Peres and Schlag). 
%Note that Mourad uses the lacunary parameter $\varepsilon$ (the ratio of denominators is at least $1+\varepsilon$) instead of our $C$, and 
Our constant $1/64$ is approximately five times better. As explained by Peres and Schlag, it is not possible to do better than $\Omega(C)$, but it is still open whether the $\log n$ term can be removed.

% \todo[inline]{Are the arguments in all the previous papers also using our weaker condition instead of lacune condition? Is $O(C\log C)$ the asymptotically best bound?
% }

% \todo[inline]{I believe that Peres and Schlag use something similar; the condition of Theorem 3.1 is $n_{j+M}\ge 2n_{M}$ which is "between" our condition and the $1+\varepsilon$ condition. Diadic intervals play the same central roles in both proofs. I've also added Bugeaud's book that give it's own version of Peres and Schalg's proof (Thm. 2.16, page 29).

% They say that a bound $O(C)$ might be possible, but we cannot hope for better than this. In fact, they mention the link with some other question by Erd\"os, which might also explain why this particular version of the problem was studied (instead of the exact original question). Katznelson (2001) is the first to mention the link between these problems.

% Let $\varepsilon>0$ and let $(S_i)_{i\ge0}$ be a lacunary sequence of parameter $(1+\varepsilon)$. Let $G=(\mathbb{Z},E)$ be the graph over vertex set $\mathbb{Z}$ where the pair $(i,j)$ is an edge iff $|n-m|\in S$. Is the chromatic number of $G$ finite?
% }

\subsection{Forbidden factors}

The questions about forbidden factors appeared several times from different perspectives (algebra, computer science, combinatorics on words). From the computer science perspective the story goes as follows. Leonid Levin noted long ago that there exist an \emph{everywhere complex sequence}, a bit sequence such the every its $n$-bit substring has Kolmogorov complexity at least $0.99n-O(1)$. (Here $0.99$ can be replaced by any constant smaller than~$1$.) Levin proved this result using Kolmogorov complexity (this proof is reproduced in~\cite{dls}). This result can be translated into a combinatorial language as follows: if for every $n$ we declare at most $2^{0.99n}$ strings as ``forbidden'', there exists a constant $c$ and an infinite sequence that does not contain forbidden strings of length greater than~$c$. (See~\cite[section 8.5, p.~241]{SUV2016} for details.) It was noted by Rumyantsev and Ushakov~\cite{RumyantsevU06} that one can prove this combinatorial statement using Lovász local lemma (LLL) and this argument can be generalized to the multidimensional case (where the original complexity argument by Levin does not work): there exists a 2-dimensional bit configuration where every rectangle has almost maximal complexity. Rumyantsev~\cite{Rumyantsev07} used this technique (Kolmogorov complexity and Lovász local lemma) for word combinatorics (construction of a sequence with given critical exponent). Miller~\cite{Miller2012} gave a much more precise condition for the existence of a sequence that avoids  a given set of forbidden strings (reproduced above as Theorem~\ref{thm:miller}) that can be easily applied to the computable case --- unlike LLL that (in its natural form) uses König's lemma that disrupts computability. However, after the Moser--Tardos effective proof of the Lovász local lemma was discovered, Fortnow suggested to look for an infinite computable version of the LLL, and indeed Rumyantsev found such a version~\cite{Rumyantsev2013,Rumyantsev2014}. The Moser--Tardos proof used some compressibility argument that can be transformed into a compressibility argument for Miller's criterion (see \cite{Shen2017,Shen2017a} for details and references) with exactly the same condition. This looks quite mysterious (why exactly the same condition appears in two completely different arguments) --- moreover, a closely related condition appeared much earlier (1964) in algebraic context (Golod--Shafarevich theorem, see~\cite[Appendix]{Shen2017a} for details), 
\ifthenelse{\boolean{anon}}{
  and reappeared in Rosenfeld's counting argument that reproves Miller's result see~\cite{Rosenfeld2020Sep}, and was also used in~\cite{Rosenfeld2022Sep} to prove that the numbers of square-free words of length $n$ that respect the $4$-list assignment grows exponentially with $n$.
}{
  and reappeared in the counting argument that reproves Miller's result  and was found by the first author, see~\cite{Rosenfeld2020Sep}. The latter argument was also used in~\cite{Rosenfeld2022Sep} where the first author of this note gave  another proof of this result based on a new counting argument which also implies that the numbers of square-free words of length $n$ that respect the $4$-list assignment grows exponentially with $n$.
}

Note also that different modification of LLL (and its computable versions) were also used to prove results about square-free words and Erdős question. In particular, Pegden in 2011 introduced his Left-handed Lov\'asz Local Lemma to study nonrepetitive games (one player wants to avoid squares or other repetitions while the other want to create them, and they alternate in adding letters to a string)~\cite{Pegden2011Jan}. In 2013, Grytczuk, Kozik and Micek used another proof of this result based on the so-called entropy compression technique, which also implies that there is an efficient algorithm to construct finite square-free words respecting $L$~\cite{Grytczuk2013Mar}.

\ifthenelse{\boolean{anon}}{
}{
\section*{Acknowledgement}
The first author would like to thank Yann Bugeaud for motivating this research by asking whether the counting argument introduced in \cite{Rosenfeld2020Sep} could be applied to the lacunary sequence problem. The answer is yes, but the current argument yields identical bounds and provide a computable result.

This research was funded in part by the French National Research Agency (ANR) under the projects ANR-24-CE48-3758-01 and ANR-21-CE48-0023.
}

\bibliographystyle{unsrt}
\bibliography{biblio}

\begin{thebibliography}{10}

\bibitem{erdosDiophantine}
Paul {Erd\H{o}s}.
\newblock {Repartition Modulo 1}.
\newblock {\em Lecture Notes in Math.}, 475, 1975.

\bibitem{Peres2010Apr}
Yuval Peres and Wilhelm Schlag.
\newblock {Two Erd{\ifmmode\mbox{\H{o}}\else\H{o}\fi}s problems on lacunary sequences: Chromatic number and Diophantine approximation}.
\newblock {\em Bull. London Math. Soc.}, 42(2):295--300, April 2010.

\bibitem{Mourad2024Jun}
Daniel Mourad.
\newblock {Computing Non-Repetitive Sequences with a Computable Lefthanded Local Lemma}.
\newblock {\em arXiv}, June 2024.

\bibitem{Miller2012}
Joseph~S. Miller.
\newblock Two notes on subshifts.
\newblock {\em Proceedings of the American Mathematical Society}, 140(5):1617--1622, 2012.

\bibitem{thue06}
Axel Thue.
\newblock {\"{U}ber unendliche {Z}eichenreihen}.
\newblock {\em Norske Vid. Selsk. Skr. I. Mat. Nat. Kl. Christiania}, 7:1--22, 1906.

\bibitem{Grytczuk2013Mar}
Jaros{\l}aw Grytczuk, Jakub Kozik, and Piotr Micek.
\newblock {New approach to nonrepetitive sequences}.
\newblock {\em Random Structures Algorithms}, 42(2):214--225, March 2013.

\bibitem{Grytczuk2011Jan}
Jaros{\l}aw Grytczuk, Jakub Przyby{\l}o, and Xuding Zhu.
\newblock {Nonrepetitive list colourings of paths}.
\newblock {\em Random Structures Algorithms}, 38(1-2):162--173, January 2011.

\bibitem{Beck1984Jan}
{J\'ozsef} Beck.
\newblock An application of {L}ov{\'a}sz local lemma: there exists an infinite $01$-sequence containing no near identical intervals.
\newblock In {\em {Finite and Infinite Sets}}, pages 103--107. North-Holland, January 1984.

\bibitem{ProbabilisticMethod}
Noga Alon, Joel~H Spencer, and Paul {Erd\H{o}s}.
\newblock {\em {The Probabilistic Method}}.
\newblock John Wiley {\&} Sons, 1992.

\bibitem{Ochem2014Apr}
Pascal Ochem and Alexandre Pinlou.
\newblock {Application of Entropy Compression in Pattern Avoidance}.
\newblock {\em Electron. J. Combin.}, page P2.7, April 2014.

\bibitem{ChybowskaSokol2022Jul}
Joanna Chybowska-Sok{\ifmmode\acute{o}\else\'{o}\fi}{\l}, Micha{\l} Debski, Jaros{\l}aw Grytczuk, Konstanty Junosza-Szaniawski, Barbara Nayar, Urszula Pastwa, and Krzysztof Wesek.
\newblock {Fractional meanings of nonrepetitiveness}.
\newblock {\em J. Combin. Theory Ser. A}, 189:105598, July 2022.

\bibitem{Bell2007Sep}
Jason~P. Bell and Teow~Lim Goh.
\newblock {Exponential lower bounds for the number of words of uniform length avoiding a pattern}.
\newblock {\em Inform. And Comput.}, 205(9):1295--1306, September 2007.

\bibitem{Rampersad2011Jun}
Narad Rampersad.
\newblock {Further Applications of a Power Series Method for Pattern Avoidance}.
\newblock {\em Electron. J. Combin.}, page P134, June 2011.

\bibitem{Ochem2016Feb}
Pascal Ochem.
\newblock {Doubled Patterns are 3-Avoidable}.
\newblock {\em Electron. J. Combin.}, page P1.19, February 2016.

\bibitem{Rosenfeld2021Oct}
Matthieu Rosenfeld.
\newblock {Lower-Bounds on the Growth of Power-Free Languages Over Large Alphabets}.
\newblock {\em Theory Comput. Syst.}, 65(7):1110--1116, October 2021.

\bibitem{Rosenfeld2022Sep}
Matthieu Rosenfeld.
\newblock {Avoiding squares over words with lists of size three amongst four symbols}.
\newblock {\em Math. Comput.}, 91(337):2489--2500, September 2022.

\bibitem{Rosenfeld2024May}
Matthieu Rosenfeld.
\newblock {Ann wins the nonrepetitive game over four letters and the erase-repetition game over six letters}.
\newblock {\em Eur. J. Combin.}, 118:103924, May 2024.

\bibitem{dls}
Bruno Durand, Leonid~A. Levin, and Alexander Shen.
\newblock Complex tilings.
\newblock {\em Journal of Symbolic Logic}, 73(2):593--613, 2008.

\bibitem{SUV2016}
Alexander Shen, Vladimir~A. Uspensky, and Nikolay Vereshchagin.
\newblock {\em Kolmogorov Complexity and Algorithmic Randomness}.
\newblock Mathematical Surveys and Monographs. American Mathematical Society, 2017.

\bibitem{RumyantsevU06}
Andrey~Yu. Rumyantsev and Maxim~A. Ushakov.
\newblock Forbidden substrings, kolmogorov complexity and almost periodic sequences.
\newblock In Bruno Durand and Wolfgang Thomas, editors, {\em {STACS} 2006, 23rd Annual Symposium on Theoretical Aspects of Computer Science, Marseille, France, February 23-25, 2006, Proceedings}, volume 3884 of {\em Lecture Notes in Computer Science}, pages 396--407. Springer, 2006.

\bibitem{Rumyantsev07}
Andrey~Yu. Rumyantsev.
\newblock Kolmogorov complexity, lov{\'{a}}sz local lemma and critical exponents.
\newblock In Volker Diekert, Mikhail~V. Volkov, and Andrei Voronkov, editors, {\em Computer Science - Theory and Applications, Second International Symposium on Computer Science in Russia, {CSR} 2007, Ekaterinburg, Russia, September 3-7, 2007, Proceedings}, volume 4649 of {\em Lecture Notes in Computer Science}, pages 349--355. Springer, 2007.

\bibitem{Rumyantsev2013}
Andrei~Yu. Rumyantsev and Alexander Shen.
\newblock Probabilistic constructions of computable objects and a computable version of lov{\'{a}}sz local lemma.
\newblock {\em CoRR}, abs/1305.1535, 2013.

\bibitem{Rumyantsev2014}
Andrei~Yu. Rumyantsev and Alexander Shen.
\newblock Probabilistic constructions of computable objects and a computable version of lov{\'{a}}sz local lemma.
\newblock {\em Fundam. Informaticae}, 132(1):1--14, 2014.

\bibitem{Shen2017}
Alexander Shen.
\newblock Compressibility and probabilistic proofs.
\newblock In Jarkko Kari, Florin Manea, and Ion Petre, editors, {\em Unveiling Dynamics and Complexity - 13th Conference on Computability in Europe, CiE 2017, Turku, Finland, June 12-16, 2017, Proceedings}, volume 10307 of {\em Lecture Notes in Computer Science}, pages 101--111. Springer, 2017.

\bibitem{Shen2017a}
Alexander Shen.
\newblock Compressibility and probabilistic proofs.
\newblock {\em CoRR}, abs/1703.03342, 2017.

\bibitem{Rosenfeld2020Sep}
Matthieu Rosenfeld.
\newblock {Another Approach to Non-Repetitive Colorings of Graphs of Bounded Degree}.
\newblock {\em Electron. J. Combin.}, page P3.43, September 2020.

\bibitem{Pegden2011Jan}
Wesley Pegden.
\newblock {Highly nonrepetitive sequences: Winning strategies from the local lemma}.
\newblock {\em Random Structures Algorithms}, 38(1-2):140--161, January 2011.

\end{thebibliography}

\end{document}